\documentclass[12pt,a4paper]{article}
\usepackage{amsmath,amssymb}
\usepackage{amssymb, amsmath, amscd}
\usepackage{comment}
\newtheorem{thm}{Theorem}[section]

\newtheorem{lmm}[thm]{Lemma}

\newtheorem{prp}[thm]{Proposition}

\newtheorem{dfn}[thm]{Definition}
\newtheorem{exa}[thm]{Example}

\newcommand{\N}{{\mathbb N}}

\newcommand{\R}{{\mathbb R}}
\newcommand{\C}{{\mathbb C}}
\newcommand{\Z}{{\mathbb Z}}
\newcommand{\ord}{\mbox{{\rm ord}}}
\newcommand{\Ann}{\mbox{{\rm Ann}}}
\newcommand{\myRe}{{\rm Re\,}\,}
\newcommand{\Fsc}{{\mathcal F}}
\newcommand{\Ksc}{{\mathcal K}}

\newcommand{\Osc}{{\mathcal O}}
\newcommand{\OC[1]}{{\mathcal O}_{\C^{#1}}}
\newcommand{\Dsc}{{\mathcal D}}
\newcommand{\DC[1]}{{\mathcal D}_{\C^{#1}}}

\newcommand{\Lsc}{{\mathcal L}}
\newcommand{\Ssc}{{\mathcal S}}
\newcommand{\Isc}{{\mathcal I}}
\newcommand{\Nsc}{{\mathcal N}}
\newcommand{\Msc}{{\mathcal M}}

\newcommand{\Char}{\mbox{{\rm Char}}}
\newcommand{\Res}{{\mathrm{Res}}}
\newcommand{\Db}{{\mathcal Db}}
\newenvironment{proof}{\noindent Proof:}{$\Box$}

\title{Operational calculus for holonomic distributions
in the framework of $D$-module theory}
\author{Toshinori Oaku}
\date{}

\begin{document}

\maketitle

\begin{abstract}
Let $f$ be a real polynomial of $x = (x_1,\dots,x_n)$ and $\varphi$ 
be a locally integrable function of $x$ which satisfies a 
holonomic system of linear differential equations. 
We study the distribution $f_+^\lambda\varphi$ with a meromorphic parameter $\lambda$, 
especially its Laurent expansion and integration, 
from an algorithmic viewpoint in the framework of $D$-module theory.
\end{abstract}

\section{Introduction}

Let $f$ be a non-constant real polynomial in $x = (x_1,\dots,x_n)$
and $\varphi$ be a locally integrable function on an open subset $U$ of $\R^n$. 
Then $\varphi$ can be regarded as a distribution (generalized function in the 
sense of L. Schwartz) on $U$. 
We assume that there exists a left ideal $I$ of the ring $D_n$ of differential 
operators with polynomial coefficients in $x$ which annihilates $\varphi$ 
on $U_f := \{x \in U \mid f(x) \neq 0\}$, i.e., $P\varphi$ vanishes on $U_f$ 
for any $P \in I$. Moreover, we assume that $M := D_n/I$ is a holonomic $D_n$-module. 
In this situation, $\varphi$ is called a (locally integrable) holonomic function 
or a holonomic distribution. 

Let us consider the distribution 
$f_+^\lambda\varphi$ on $U$ with a holomorphic parameter $\lambda$. 
This distribution can be analytically extended to a distribution-valued 
meromorphic function of $\lambda$ on the complex plane $\C$. 
Such a distribution was systematically studied by Kashiwara and Kawai 
in $\cite{KK}$ with $f$ being, more generally, a real-valued real analytic function. 
Their investigation was focused on a special case where $M$ has regular singularities 
but most of the arguments work without this assumption. 

The main purpose of this article is to give algorithms to compute
\begin{enumerate}
\item
A holonomic system for the distribution $f_+^{\lambda_0}\varphi$ with $\lambda_0$ 
not being a pole of $f_+^\lambda\varphi$. 
\item
A holonomic system for each coefficient of the 
Laurent series of $f_+^\lambda\varphi$ about an arbitrary point.
\item
Difference equations for the local zeta function 
$Z(\lambda) = \int_{\R^n}f_+^\lambda\varphi\,dx$. 
\end{enumerate}
As was pointed out in \cite{KK}, 
an answer to the first problem provides us with an algorithm 
to compute a holonomic system for the product of two locally $L^2$ holonomic 
functions. Note that the product does not necessarily satisfies 
the tensor product of the two holonomic systems for both functions. 

In Section 2, we review the theoretical properties of $f_+^\lambda\varphi$ 
mostly following Kashiwara \cite{K2} and Kashiwara and Kawai \cite{KK} in the 
analytic category; i.e, under a weaker assumption that $f$ is a real-valued 
real analytic function and that $\varphi$ satisfies a holonomic system of 
linear differential equations with analytic coefficients. 

In Section 3, we give algorithms to computes holonomic systems considered in 
Section 2. As a byproduct, we obtain an algorithm to compute difference equations 
for the local zeta function, which was outlined in \cite{O2}.

\section{Theoretical background}


Let $\DC[n]$ be the sheaf on $\C^n$ of linear partial differential 
operators with holomorphic coefficients, which is generated by 
the derivations $\partial_j = \partial_{x_j} = \partial/\partial x_j$ 
$(j=1,\dots,n)$ over the sheaf $\OC[n]$ of rings of holomorphic functions 
on $\C^n$, with the coordinate system $x = (x_1,\dots,x_n)$ of $\C^n$.
 
We denote by $\Db$ the sheaf on $\R^n$ of the Schwartz distributions. 
Assume that
$f = f(x)$ is a nonzero real-valued real analytic function 
defined on an open connected set $U$ of $\R^n$. 
Let 
$\varphi$ be a locally integrable function on $U$.  
Then $f_+^{\lambda}\varphi$ is also locally integrable on $U$ 
for any $\lambda \in \C$ with $\myRe \lambda \geq 0$, where 
$f_+(x) = \max\{f(x),0\}$.  

Let $\Msc$ be a holonomic  $\DC[n]$-module defined on an 
open set $\Omega$ of $\C^n$ such that $U \subset \Omega \cap \R^n$. 
We say that a distribution $\varphi$ is a solution of $\Msc$ on $U$  
if there exist a section  $u$ of $\Msc$ on $U$ and a $\DC[n]$-linear 
homomorphism $\Phi : \DC[n]u \rightarrow \Db$ defined on $U$ 
such that $\Phi(u) = \varphi$.
As a matter of fact, we have only to assume that $\varphi$ is a solution 
of $\Msc$ on $U_f := \{x \in U \mid f(x) \neq 0\}$ 
and that $\Msc$ is holonomic on $\Omega_f := \{x \in \Omega \mid f(x) \neq 0\}$.

\subsection{Fundamental lemmas}

Under the assumptions above, 
$f_+^\lambda\varphi$ is a $\Db(U)$-valued holomorphic function of $\lambda$  
on the right half-plane
\[
\C_+ := \{\lambda \in \C \mid \myRe\lambda > 0 \}.
\]
In other words, 
let $\Osc\Db$ be the sheaf on $\C \times \R^n \ni (\lambda,x)$
of distributions 
with a holomorphic parameter $\lambda$. 
Then 
$f_+^\lambda\varphi$ belongs to 
\[
\Osc\Db(\C_+ \times U)
= \Bigl\{ v(\lambda,x) \in \Db(\C_+ \times U) \mid \frac{\partial v}{\partial 
\overline{\lambda}} = 0\Bigr\}.
\] 
Let $s$ be an indeterminate corresponding to $\lambda$. 
The following lemma (Lemma 2.9 of \cite{KK}) plays an essential 
role in the following arguments. 

\begin{lmm}[Kashiwara-Kawai \cite{KK}]\label{lemma:fundamental}
Let $\Omega$ be an open set of $\C^n$ such that $V := \R^n \cap \Omega$ 
is non-empty. 
Assume $P(s) \in \Dsc_{\C^n}(\Omega)[s]$ and  
$P(\lambda)(f_+^{\lambda}\varphi) = 0$ 
holds in $\Osc\Db(\C_+ \times V_f)$ with 
$V_f := \{x \in V \mid f(x) \neq 0\}$.  Then 
$P(\lambda)(f_+^{\lambda}\varphi) = 0$ holds 
in $\Osc\Db(\C_+ \times V)$.
\end{lmm}

Let us generalize this lemma slightly. 
For a positive integer $m$, let us define a section
$f_+^\lambda(\log f_+)^m\varphi$  of the sheaf $\Osc\Db$ on $\C_+ \times U$ by
\[
 \langle f_+^\lambda(\log f_+)^m\varphi, \psi\rangle 
= \int_{\{x \in U \mid f(x) > 0\}} 
\varphi(x)f(x)^\lambda(\log f(x))^m \varphi(x)\psi(x)\,dx 
\quad (\forall \psi \in C_0^\infty(U)),
\]
where $C_0^\infty(U)$ denotes the space of $C^\infty$ functions on $U$ with 
compact supports. 
In fact, $f_+^\lambda(\log  f_+)^m\varphi$ is the $m$-th derivative 
of the distribution $f_+^\lambda\varphi$ with respect to $\lambda$. 

\begin{lmm}\label{lemma:fundamental2}
Let $\Omega$ be an open set of $\C^n$ such that $V := \R^n \cap \Omega$ 
is non-empty. 
Let $\varphi_0,\dots,\varphi_m$ be locally integrable functions on $V$. 
Assume $P_k(s) \in \Dsc_{\C^n}(\Omega)[s]$ ($k=0,1,\dots,m$) and   
\begin{equation}\label{eq:fslog}
\sum_{k=0}^m P_k(\lambda)(f_+^{\lambda}(\log f_+)^k\varphi_k) = 0
\end{equation}
holds in $\Osc\Db(\C_+ \times V_f)$. 
Then (\ref{eq:fslog}) holds in $\Osc\Db(\C_+\times V)$. 
\end{lmm}

\begin{proof}
We follow the argument of the proof of Lemma 2.9 in \cite{KK}. 
Let $\phi$ belong to $C^\infty_0(V)$ with $K := \mathrm{supp}\, \phi$. 
Let $\chi(t)$ be a $C^\infty$ function of a variable $t$ 
such that $\chi(t) = 1$ for $|t| \leq 1/2$ and $\chi(t) = 0$ for $|t| \geq 1$. 
Then we have
\begin{align*}
\left\langle
\sum_{k=0}^m P_k(\lambda)(f_+^{\lambda}(\log f_+)^k\varphi_k),\,\,\phi\right\rangle
&= 
\left\langle\sum_{k=0}^m  P_k(\lambda)(f_+^{\lambda}(\log f_+)^k\varphi_k),\,\,
\chi\bigl(\frac{f}{\tau}\bigr) \phi\right\rangle
\\
&=
\sum_{k=0}^m \int_V f_+^{\lambda}(\log f_+)^k\varphi_k {}^tP_k(\lambda)
\Bigl(\chi\bigl(\frac{f}{\tau}\bigr)\phi\Bigr)\,dx
\end{align*}
for any $\tau > 0$, where ${}^tP_k(\lambda)$ denotes the adjoint operator of 
$P_k(\lambda)$. 
Let $m_k$ be the order of $P_k(s)$ and $d_k$ be the degree of $P_k(s)$ in $s$. 
Then there exist constants $C_k$ such that
\[
\sup_{x \in K}\left|{}^tP_k(\lambda)\Bigl(\chi\bigl(\frac{f(x)}{\tau}\bigr)\phi(x)\Bigr)\right|
\leq C_k(1+|\lambda|)^{d_k}\tau^{-m_k} 
\qquad (0 < \forall \tau < 1).
\]
Assume $\myRe\lambda > \max\{m_k+1 \mid 0 \leq k \leq m\}$ 
and $0 < \tau < 1$. Then we have
\begin{align*}
&
\left|\int_V f_+^{\lambda}(\log f_+)^k\varphi_k {}^tP_k(\lambda)
\Bigl(\chi\bigl(\frac{f}{\tau}\bigr)\phi\Bigr)\,dx \right| 
\\&
\leq C_k(1+|\lambda|)^{d_k}\tau^{-m_k} 
\int_{\{x\in V | 0 < f(x) \leq \tau\}} |f_+^{\lambda}(\log f_+)^k\varphi_k(x)|\,dx 
\\&
\leq k!C_k(1+|\lambda|)^{d_k}\tau^{\myRe\lambda - m_k - 1}
\int_{\{x\in V | 0 < f(x) \leq \tau\}} |\varphi_k(x)|\,dx
\end{align*}
since $|\log t|^k \leq k! t^{-1}$ holds for $0 < t < 1$.  
This implies
\begin{align*}
\left\langle
\sum_{k=0}^m P_k(\lambda)(f_+^{\lambda}(\log f_+)^k\varphi_k),\,\,\phi\right\rangle
&=
\lim_{\tau\rightarrow +0}
\sum_{k=0}^m \int_V f_+^{\lambda}(\log f_+)^k\varphi_k {}^tP_k(\lambda)
\Bigl(\chi\bigl(\frac{f}{\tau}\bigr)\phi\Bigr)\,dx
=0.
\end{align*}
The assertion of the lemma follows from the uniqueness of analytic 
continuation.
\end{proof}

\subsection{Generalized $b$-function and analytic continuation}

We assume that there exists on $\Omega$ 
a sheaf $\Isc$ of coherent left ideals of 
$\DC[n]$ which annihilates $\varphi$ on $U_f = \{x \in U \mid f(x) \neq 0\}$, 
namely, $P\varphi=0$ holds on $W \cap U_f$ for any section $P$ of $\Isc$ on 
an open set $W$ of $\C^n$. 
We set $\Msc = \DC[n]/\Isc$ and denote by $u$ the residue class of 
$1 \in \Dsc_X$ modulo $\Isc$. 
In the sequel, we assume that $\Msc$ is holonomic on 
$\Omega_f = \{z \in \Omega \mid f(z) \neq 0\}$, i.e., 
that $\Char(\Msc)\cap \pi^{-1}(\Omega_f)$ is of dimension $n$,
where $\Char(\Msc)$ denotes the characteristic variety of $\Msc$ 
and $\pi :T^*\C^n \rightarrow \C^n$ is the canonical projection. 

Let 
$\Lsc = \OC[n][f^{-1},s]f^s$
be the free $\OC[n][f^{-1},s]$-module generated by the symbol $f^s$. 
Then $\Lsc$ has a natural structure of left $\DC[n][s]$-module
induced by the derivation 
$\partial_i f^s = s(\partial f/\partial x_i)f^{-1}f^s$.  
Let us consider the tensor product $\Lsc \otimes_{\OC[n]} \Msc$ 
of $\OC[n]$-modules, which has a natural structure of left $\DC[n][s]$-module. 

\begin{lmm}\label{lemma:saturation}
Let $v$ and $P(s)$ be sections of $\Msc$ and of $\DC[n][s]$ 
respectively on an open subset of $\Omega$. Then 
$P(s)(f^s\otimes v) = 0$ holds in $\Lsc\otimes_{\OC[n]}\Msc$
if and only if 
$(f^{m-s}P(s)f^s)(1\otimes v) = 0$ holds in $\C[s]\otimes_\C \Msc$ 
for a sufficiently large $m \in \N$. 
\end{lmm}

\begin{proof}
Set $\Msc[s] = \C[s]\otimes_\C \Msc$, which has a natural structure of 
left module over $\C[s]\otimes_\C\DC[n] = \DC[n][s]$.  Then we have 
$\Lsc\otimes_{\OC[n]}\Msc = \Lsc\otimes_{\OC[n][s]}\Msc[s]$ as 
left $\DC[n][s]$-module. 
Let $v$ be a section of $\Msc[s]$. 
Since $\Lsc$ is isomorphic to $\OC[n][f^{-1},s]$ as $\OC[n][s]$-module, 
$f^s\otimes v$ vanishes in $\Lsc\otimes_{\OC[n][s]}\Msc[s]$ if and only if 
$1\otimes v$ vanishes in $\OC[n][f^{-1},s]\otimes_{\OC[n][s]}\Msc[s]$. 
First, let us show that this happens if and only if $f^mv=0$ in $\Msc[s]$ 
with some $m \in \N$. 

Let $\rho : \OC[n][s,t] \rightarrow \OC[n][s,f^{-1}]$ be the 
homomorphism defined by $\rho(h(s,t)) = h(s,f^{-1})$ for 
$h(s,t) \in \OC[n][s,t]$. Let $\Ksc$ be the kernel of $\rho$. 
Then we have an exact sequence
\[
\Ksc\otimes_{\OC[n][s]}\Msc[s] \longrightarrow
\OC[n][s,t]\otimes_{\OC[n][s]}\Msc[s] 
\stackrel{\rho\otimes\mathrm{id}}{\longrightarrow} 
\OC[n][s,f^{-1}]\otimes_{\OC[n][s]}\Msc[s] 
\longrightarrow 0. 
\] 
Hence $1\otimes v$ vanishes in $\OC[n][s,f^{-1}]\otimes_{\OC[n][s]}\Msc[s]$ 
if and only if there exists $h(s,t) = \sum_{k=0}^m h_k(s)t^k\in \Ksc$ such that 
$1\otimes v = h(s,t)\otimes v$ holds in $\OC[n][s,t]\otimes_{\OC[n][s]}\Msc[s]$, 
which is equivalent to $h_k(s)v = \delta_{0k}v$ ($k=0,1,\dots,m$)   
since $\OC[n][s,t]$ is free over $\OC[n][s]$. 
On the other hand, $\sum_{k=0}^m h_k(s)f^{-k} = \rho(h(s,t)) = 0$ implies
\[
0 =  f^mh_0(s)v + f^{m-1}h_1(s)v + \cdots + fh_{m-1}(s)v + h_m(s)v = f^mv.  
\]
Conversely, if $f^mv=0$ for some $m \in \N$, then we have 
$1\otimes v = f^{-m}\otimes f^m v=0$ in $\OC[n][s,f^{-1}]\otimes_{\OC[n][s]}\Msc$. 

Let $P(s)$ be a section of $\DC[n][s]$ of order $m$. 
For $i=1,\dots,n$, 
\[
\partial_i(f^s\otimes v) = f^{s-1}\otimes (sf_i + f\partial_i)v
= f^{s-1}\otimes(f^{1-s}\partial_if^s)v
\]
holds in $\Lsc\otimes_{\OC[n][s]}\Msc[s]$ 
with $f_i = \partial f/\partial x_i$. 
This allows us to show that
\[
P(s)(f^s\otimes v) = f^{s-m}\otimes (f^{m-s}P(s)f^s)v
\]
holds in $\Lsc\otimes_{\OC[n][s]}\Msc[s]$. 
(Note that $f^{m-s}P(s)f^s$ belongs to $\DC[n][s]$.) 
Summing up, we have shown that $P(s)(f^s\otimes v)$ vanishes 
in $\Lsc\otimes_{\OC[n][s]}\Msc[s]$ if and only if 
$(f^{l-s}P(s)f^s)v$ vanishes in $\Msc[s]$ for some $l \geq m$. 
\end{proof}

Lemma \ref{lemma:saturation} with $P(s) = 1$ immediately implies

\begin{prp}
Let $\Msc[f^{-1}] := \OC[n][f^{-1}]\otimes_{\OC[n]}\Msc$ be the 
localization of $\Msc$ with respect to $f$, which has a natural 
structure of left $\DC[n]$-module. 
Then the natural homomorphism 
$\Lsc\otimes_{\OC[n]}\Msc \rightarrow \Lsc\otimes_{\OC[n]}\Msc[f^{-1}]$ 
is injective. 
\end{prp}

\begin{prp}\label{prop:fundamental}
Let $P(s)$ be a section of $\DC[n][s]$ on an open set $\Omega$ of $\C^n$ and suppose 
$P(s)(f^s\otimes u) = 0$ in $\Lsc\otimes_{\OC[n]}\Msc$. 
Set $V = U \cap \Omega$. 
Then $P(\lambda)(f_+^\lambda\varphi) = 0$ 
holds in $\Osc\Db(\C_+ \times V)$.
\end{prp}

\begin{proof}
Let $\Osc_{+\infty}\Db$ be the sheaf on $\R^n$ associated with 
the presheaf 
\[
W \longmapsto  \lim_{\longrightarrow} 
\Osc\Db(\{\lambda\in\C \mid \myRe\lambda > a\} \times W)
\]
for every open set $W$ of $\R^n$,
where the inductive limit is taken as $a \rightarrow \infty$.  
The $\C$-bilinear sheaf homomorphism 
\[
\Lsc  \times \Msc \ni 
(a(s)f^{s-m}, Pu) \longmapsto (a(\lambda)f_+^{\lambda-m})P\varphi 
\in \Osc_{+\infty}\Db
\]
with $a(s) \in \Osc_X[s]$, $m \in \N$, $P \in \Dsc_X$, 
which is well-defined and $\OC[n]$-balanced on $V_f$  
since $f_+^{\lambda-m}$ is real analytic there, 
induces a $\DC[n]$-linear homomorphism
\[
\Psi :\,\Lsc \otimes_{\OC[n]}\Msc \longrightarrow \Osc_{+\infty}\Db
\]
on $V_f$
such that $\Psi(a(s)f^{s-m}\otimes Pu) = a(\lambda)f_+^{\lambda-m}P\varphi$. 
In particular, 
if $P(s) \in \DC[n][s]$ satisfies 
$P(s)(f^s\otimes u) = 0$ in $\Lsc\otimes_{\OC[n]}\Msc$, then
$P(\lambda)(f_+^\lambda \varphi) = 0$ holds in $\Osc_{+\infty}\Db(V_f)$, 
hence also in $\Osc_{+\infty}\Db(V)$ by Lemma \ref{lemma:fundamental}. 
Since $f_+^\lambda\varphi$ belongs to 
$\Osc\Db(\C_+ \times V)$, 
it follows that $P(f_+^\lambda\varphi) = 0$ holds in $\Osc\Db(\C_+ \times V)$. 
This completes the proof.
\end{proof}

Kashiwara proved in \cite{K2} (Theorem 2.7) that 
on a neighborhood of each point $p$ of $\Omega$,  
there exist nonzero $b(s) \in \C[s]$ and $P(s) \in \DC[n][s]$ such that
\[
P(s) (f^{s+1} \otimes u) = b(s)f^s\otimes u 
\quad\mbox{in $\Lsc\otimes_{\OC[n]}\Msc$}. 
\]
Such $b(s)$ of the smallest degree $b(s) = b_p(s)$ is called the (generalized) 
$b$-function for $f$ and $u$ at $p$. 

Assume $p \in U$. Then 
by the proposition above, 
\[
P(\lambda)(f_+^{\lambda+1}\varphi) = b(\lambda)f_+^\lambda\varphi
\]
holds in $\Osc\Db(\C_+ \times V)$ with an open neighborhodd $V$ of $p$. 
It follows that $f_+^\lambda\varphi$ is a $\Db(V)$-valued 
meromorphic function of $\lambda$ on $\C$. 
It is easy to see that we can replace $V$ by an arbitrary relatively compact 
subset of $U$. 
The poles of $f_+^\lambda\varphi$ are contained in  
\[
\{\lambda -k \mid b_p(\lambda) =0\,\,(\exists p \in V),\, k \in \N\}.
\]

\begin{prp}[Lemma 2.10 of \cite{KK}] \label{prop:roots}
There exists a positive real number $\varepsilon$ such that 
$f_+^\lambda\varphi$ belongs to 
$\Osc\Db(\{\lambda \in \C \mid \myRe\lambda > -\varepsilon\} \times U)$. 
\end{prp}

\begin{proof}
Let $\lambda_0$ be an arbitrary pole of $f_+^\lambda\varphi$. 
There exists $\psi \in C_0^\infty(U)$ such that $\lambda_0$ is 
a pole of $Z(\lambda):= \langle f_+^\lambda\varphi,\psi\rangle$. 
In particular, $|Z(\lambda_0+t)|$ tends to infinity as $t \rightarrow +0$. 
On the other hand, $Z(\lambda)$ is continuous on 
$\{\lambda \in \C \mid \myRe\lambda \geq 0\}$. 
This implies $\myRe \lambda_0 < 0$. 
The conclusion follows since there are at most a finite number of 
poles of $f_+^\lambda\varphi$ in the set 
$\{\lambda \in \C | \myRe\lambda > -1\}$. 
\end{proof}

In conclusion, $f_+^\lambda\varphi$ is a $\Db(U)$-valued meromorphic 
function on $\C$ whose poles are contained in 
$\{\lambda \in \C \mid \myRe\lambda < 0\}$.

\subsection{Holonomicity of $f_+^\lambda\varphi$ and its applications}

Let $f$, $\varphi$, $\Msc = \DC[n]/\Isc$ be as in the preceding subsection. 
Let $\Nsc = \DC[n][s](f^s\otimes u)$ be the left 
$\DC[n][s]$-submodule of $\Lsc\otimes_{\OC[n]}\Msc$ 
generated by $f^s\otimes u$. 
Theorem 2.5 of Kashiwara \cite{K2} guarantees that 
$\Nsc_{\lambda_0} := \Nsc/(s-\lambda_0)\Nsc$ 
is a holonomic $\DC[n]$-module on $\Omega$ for any $\lambda_0 \in \C$.

\begin{prp}
Let $\lambda_0$ be an arbitrary complex number
and $f^{\lambda_0}\otimes\varphi$ the residue class of 
$f^s\otimes u \in \Nsc$ modulo $(s-\lambda_0)\Nsc$.  
\begin{enumerate}
\item
$\Nsc_0$ is isomorphic to $\Msc$ as $\DC[n]$-module on $\Omega_f$. 
\item
If $\Msc$ is $f$-saturated, i.e., if $fv = 0$ with $v \in \Msc$ 
implies $v= 0$, then
there is a surjective $\DC[n]$-homomorphism 
$\Phi : \Nsc_0 \rightarrow \Msc$ on $\Omega$ 
such that $\Phi(f^0\otimes u) = u$. 
Moreover, $\Phi$ is an isomorphism on $\Omega_f$. 
\end{enumerate}
\end{prp}

\begin{proof}
Since $\Msc[f^{-1}]= \Msc$ on $\Omega_f$, 
we may assume that $\Msc$ is $f$-saturated. 
In view of Lemma \ref{lemma:saturation} and the definition of 
$\Nsc_0$,  $P \in \DC[n]$ annihilates $f^{0}\otimes u$ if and only if 
there exist $Q(s) \in \DC[n][s]$ and an integer $m \geq \ord\, Q(s)$ such that
$(f^{m-s}Q(s)f^s)(1\otimes u) = 0$ in $\Msc[s]$ and  
$P = Q(0)$. 
If there exist such $Q(s)$ and $m$, set
\[
 f^{m-s}Q(s)f^s = Q_0 + Q_1s + \cdots + Q_m s^m 
\quad (Q_i \in \DC[n]).
\]
Then $Q_iu=0$ holds for any $i$.  In particular, 
$Q_0 = f^mP$ annihilates $u$. This implies $Pu=0$ since $\Msc$ is $f$-saturated. 
Hence the homomorphism $\Phi$ is well-defined. 

Now assume $f \neq 0$ and $Pu=0$. 
Then $Q(s) := f^{s}Pf^{-s}$ belongs to $\DC[n][s]$ 
and annihilates $f^s\otimes u$ by Lemma \ref{lemma:saturation}. 
Hence $P = Q(0)$ annihilates $f^0\otimes u$. 
This implies that $\Phi$ is an isomorphism on $\Omega_f$. 
\end{proof}

\begin{thm}\label{th:fsu}
If $\lambda_0$ is not a pole of $f_+^\lambda\varphi$, 
then $f_+^{\lambda_0}\varphi$ is a solution of $\Nsc_{\lambda_0}$. 
\end{thm}

\begin{proof}
Assume that $\lambda_0 \in \C$ is not a pole of $f_+^\lambda\varphi$. 
Let $P$ be a section of $\DC[n]$ which annihilates $f^{\lambda_0}\otimes u$. 
Then there exist $Q(s),R(s) \in \DC[n][s]$ such that
\[
 P = Q(s) + (s-\lambda_0)R(s), 
\quad
 Q(s)(f^s\otimes u) = 0 \mbox{ in $\Nsc$}.
\] 
Proposition \ref{prop:fundamental} implies that 
$Q(\lambda)(f_+^\lambda\varphi)$ vanishes as section of the sheaf 
$\Osc\Db$. In particular, 
$P(f_+^{\lambda_0}\varphi) = Q(\lambda_0)(f_+^{\lambda_0}\varphi) = 0$ 
holds as distribution. 
Thus the homomorphism 
\[
\DC[n](f^{\lambda_0}\otimes u) \ni P(f^{\lambda_0}\otimes u) 
\longmapsto P(f_+^{\lambda_0}\varphi) \in \Db
\]
is well-defined and $\DC[n]$-linear. 
Hence $f_+^{\lambda_0}\varphi$ is a solution of $\Nsc_{\lambda_0}$. 
\end{proof}

The following two theorems are essentially due to Kashiwara and Kawai 
\cite{KK} although they are stated with additional assumptions and stronger results. 

\begin{thm}\label{th:extension}
$\varphi$ is a soution of the holonomic $\DC[n]$-module $\Nsc_0$. 
\end{thm}

\begin{proof}
First note that 
$\OC[n][f^{-1},s](-f)^s$ is isomorphic to $\OC[n][f^{-1},s]f^s$ 
as left $\DC[n][s]$-module 
since $\partial_i(-f)^s = sf_if^{-1}(-f)^{s}$ holds in 
$\OC[n][f^{-1},s](-f)^s$ with $f_i = \partial f/\partial x_i$. 
Assume that $P(f^0\otimes u) = 0$ holds in $\Nsc_0 = \Nsc/s\Nsc$. 
Then there exist $Q(s),R(s) \in \DC[n][s]$ such that
\[
 P = Q(s) + sR(s), 
\quad
 Q(s)(f^s\otimes u) = 0 \mbox{ in $\Nsc$}.
\] 
Let $\theta(t)$ be the Heaviside function; i.e., 
$\theta(t) = 1$ for $t > 0$ and $\theta(t) = 0$ for $t \leq 0$.  
Then we have $\theta(f) = f_+^0$ and $\theta(-f) = (-f)_+^0$. 
Theorem \ref{th:fsu} implies that $P = Q(0)$ annihilates both 
$\theta(f)\varphi$ and $\theta(-f)\varphi$, 
and hence also $\varphi = \theta(f)\varphi + \theta(-f)\varphi$. 
Thus $\varphi$ is a solution of $\Nsc_0$. 
\end{proof}

\begin{thm}
Let $\varphi_1$ and $\varphi_2$ be locally $L^p$ and $L^q$ functions 
respectively on an open set $U \subset \R^n$  
with $1 \leq p,q \leq \infty$ and $1/p + 1/q = 1$. 
Assume that $\varphi_1$ and $\varphi_2$ 
are solutions of holonomic $\Dsc_{\C^n}$-modules $\Msc_1$ and $\Msc_2$ respectively on $U$.  
Then for any point $x_0$ of $U$, 
there exists a holonomic $\Dsc_{\C^n}$-module $\Msc$ on a neighborhood of $x_0$ 
of which the product $\varphi_1\varphi_2$ is a solution. 
\end{thm}

\begin{proof}
There exist analytic functions $f_1$ and $f_2$ on a neighborhood $V$ of 
$x_0$  such that 
the singular support (the projection of the characteristic variety 
minus the zero section) of $\Msc_k$ is contained in $f_k=0$ 
for $k=1,2$. Set $f(z) = f_1(z)\overline{f_1(\overline z)} 
f_2(z)\overline{f_2(\overline z)}$. Then $f(x)$ is a real-valued 
real analytic function 
and $\varphi_1$ and $\varphi_2$ are real analytic on $V_f$. 
Then it is easy to see, in the same way as in the proof of Theorem \ref{th:fsu}, 
that $\varphi_1\varphi_2$ is a solution of $\Msc_1\otimes_{\OC[n]}\Msc_2$ 
on $V_f$. To complete the proof, we have only to apply  Theorem \ref{th:extension} 
to $\Msc_1\otimes_{\OC[n]}\Msc_2$ and $f$. 
\end{proof}

\subsection{Laurent coefficients of $f_+^\lambda\varphi$}

Let $f$, $\varphi$, $\Msc$ be as in preceding subsections.

\begin{thm}
Let $p$ be a point of $U$. 
Then each coefficient of 
the Laurent expansion of $f_+^\lambda\varphi$ about an arbitrary 
$\lambda_0 \in \C$  
is a solution of a holonomic $\Dsc_{\C^n}$-module on a common neighborhood of $p$. 
\end{thm}

\begin{proof}
Fix $m \in \N$ such that $\myRe \lambda_0 + m > 0$. 
By using the functional equation involving the generalized $b$-function, 
we can find a nonzero $b(s) \in \C[s]$ and 
a germ $P(s)$ of $\DC[n][s]$ at $p$ such that
\[
 b(\lambda)f_+^\lambda\varphi = P(\lambda)(f_+^{\lambda+m}\varphi). 
\]
Factor $b(s)$ as $b(s) = (s-\lambda_0)^l c(s)$ with $c(s) \in \C[s]$ 
such that $c(\lambda_0) \neq 0$ and an integer $l \geq 0$.  
Then we have
\[
(\lambda - \lambda_0)^l f_+^\lambda\varphi
= \frac{1}{c(\lambda)}P(\lambda)(f_+^{\lambda+m}\varphi). 
\]
The right-hand side is holomorphic in $\lambda$ on an neighborhood of 
$\lambda = \lambda_0$. 
Let 
\[
f_+^\lambda\varphi 
= \sum_{k=-l}^\infty (\lambda - \lambda_0)^k \varphi_{k} 
\]
be the Laurent expansion with $\varphi_k \in \Db(U)$, which is given by
\[
\varphi_k = \frac{1}{(l+k)!} \lim_{\lambda \rightarrow \lambda_0}
\frac{\partial^{l+k}}{\partial \lambda^{l+k}}
\left((\lambda - \lambda_0)^lf_+^\lambda\varphi\right)
= 
\frac{1}{(l+k)!} \lim_{\lambda \rightarrow \lambda_0}
\frac{\partial^{l+k}}{\partial \lambda^{l+k}}
\left(\frac{1}{c(\lambda)}P(\lambda)(f_+^{\lambda+m}\varphi)\right)
.
\]
Hence there exist $Q_{kj} \in \DC[n]$ such that
\begin{equation}\label{eq:phik}
\varphi_k = \sum_{j=0}^{l+k} Q_{kj}(f_+^{\lambda_0+ m}(\log f_+)^j\varphi). 
\end{equation}

First let us show that $f_+^{\lambda_0+ m}(\log f_+)^j\varphi$ 
with $0 \leq j \leq k$ satisfy a holonomic system.
Consider the free $\OC[n][s,f^{-1}]$-module 
\[
 \tilde\Lsc := \OC[n][s,f^{-1}]f^s \oplus \OC[n][s,f^{-1}]f^s \log f 
\oplus \OC[n][s,f^{-1}] f^s (\log f)^2 \oplus \cdots,  
\]
which has a natural structure of left $\DC[n][s]$-module. 
Let
\[
 \Nsc[k] := \DC[n][s](f^s\otimes u) + \DC[n][s]((f^s\log f)\otimes u) + \cdots 
+ \DC[n][s]((f^s(\log f)^k)\otimes u)
\]
be the left $\DC[n][s]$-submodule of $\tilde\Lsc \otimes_{\OC[n]}\Msc$
generated by $(f^s(\log f)^j)\otimes u$ with $j=0,1,\dots,k$. 
It is easy to see that $\Nsc[k]/\Nsc[k-1]$ is isomorphic to 
$\Nsc = \Nsc[0]$ as left $\DC[n][s]$-module since 
\[
 P(s)((f^s(\log f)^k)\otimes u) \equiv 
 (f^{s-m}(\log f)^k)\otimes (f^{m-s}P(s)f^s) u \quad \mod \Nsc[k-1]
\] 
holds for any $P(s) \in \DC[n][s]$ with $m = \ord\, P(s)$. 
Moreover, $\Nsc_{\lambda_0}[k] := \Nsc[k]/(s-\lambda_0)\Nsc[k]$ 
is a holonomic $\DC[n]$-module since 
$\Nsc_{\lambda_0}[k]/\Nsc_{\lambda_0}[k-1]$ is isomorphic to 
$\Nsc_{\lambda_0} = \Nsc_{\lambda_0}[0]$, and hence is holonomic 
as left $\DC[n]$-module. 

Let $(f^{\lambda_0+m}(\log f)^j)\otimes u \in \Nsc_{\lambda_0+m}[k]$ 
be the residue class of $(f^s(\log f)^j)\otimes u$ modulo 
$(s-\lambda_0-m)\Nsc[k]$. 
Suppose 
$\sum_{j=0}^k P_j((f^{\lambda_0+m}(\log f)^j)\otimes u)$ vanishes 
in $\Nsc_{\lambda_0+m}[k]$ with $P_j$ being a section of $\DC[n]$ 
on an open neighborhood of a point $p$ of $U$. 
Then there exist $Q_j(s) \in \DC[n][s]$ such that
\[
\sum_{j=0}^k P_j((f^s(\log f)^j) \otimes u) 
= (s-\lambda_0-m)\sum_{j=0}^k Q_j(s)((f^s(\log f)^j) \otimes u)
\]
holds in $\Nsc[k]$. 
Then it is easy to see that 
\begin{equation}\label{eq:log}
\sum_{j=0}^{k} P_j(\lambda)(f_+^\lambda(\log f_+ )^j\varphi) 
= (\lambda-\lambda_0-m)\sum_{j=0}^k Q_j(\lambda)(f_+^\lambda(\log f_+)^j\varphi)
\end{equation}
holds in $\Osc\Db(\C_+ \times W_f)$ with an open neighborhood $W$ of $p$.  
Lemma \ref{lemma:fundamental2} and analytic continuation imply that 
(\ref{eq:log}) holds in $\Osc\Db(\C_+ \times W)$. 
Hence we have in $\Db(W)$
\[
\sum_{j=0}^{k} P_j((f_+^{\lambda_0+m}(\log f_+)^j) \varphi) = 0 .
\]

In conclusion, with $k$ replaced by $l+k$, there exists a $\DC[n]$-homomorphism 
$\Phi : \Nsc_{\lambda_0+m}[l+k] \rightarrow \Db$ such that
\[
\Phi((f^{\lambda_0+m}(\log f)^j)\otimes u) = 
 f_+^{\lambda_0+m}(\log f_+)^j 
\qquad (0 \leq j \leq l+k).
\]
Set 
\[
w := \sum_{j=0}^{l+k} Q_{kj}((f^{\lambda_0+m}(\log f)^j)\otimes u), 
\qquad
\Msc_{k} := \DC[n]w.
\]
Then $\Msc_{k}$ is a $\DC[n]$-submodule of $\Nsc_{\lambda_0+m}[l+k]$ and hence holonomic. 
Since $\Phi(w) = \varphi_k$ in view of (\ref{eq:phik}), 
$\varphi_k$ is a solution of $\Msc_k$. 
This completes the proof. 
\end{proof}

\section{Algorithms}

We give algorithms for computing holonomic systems introduced in 
the previous section assuming that $f$ is a real polynomial and 
that $\Msc$ is algebraic, i.e., defined by differential 
operators with polynomial coefficients. 
Let $D_n := \C\langle x,\partial\rangle 
= \C\langle x_1,\dots,x_n,\partial_1,\dots,\partial_n\rangle$ 
be the ring of differential operators with polynomial coefficients 
with $\partial_j = \partial/\partial x_j$. 
The ring $D_n$ is also called the $n$-th Weyl algebra over $\C$. 

In the sequel, let $f$ be a non-constant real polynomial of $x = (x_1,\dots,x_n)$ 
and $\varphi$ be a locally integrable function on an open connected set $U$ 
of $\R^n$. 
We assume that there exists a left ideal $I$ of $D_n$ which annihilates 
$\varphi$ on $U_f$, i.e., $P\varphi = 0$ holds on $U_f$ for any $P \in I$, 
such that $M := D_n/I$ is a holonomic $D_n$-module. 
We denote by $u$ the residue class of $1 \in D_n$ modulo $I$. 
Let $L = \C[x,f^{-1},s]f^s$ be the free $\C[x,f^{-1},s]$-module generated 
by $f^s$, which has a natural structure of left $D_n[s]$-module. 
Let $N := D_n[s](f^s\otimes u)$ be the left $D_n$-submodule of 
$L\otimes_{\C[x]}M$ generated by $f^s\otimes u$. 

As was established in the previous section, 
$f_+^\lambda\varphi$ is a $\Db(U)$-valued meromorphic function 
on $\C$ and is a solution of $N$. 

\subsection{Mellin transform}

Let us assume that $\varphi$ is real analytic on $U_f$ and 
set
\[
\tilde\varphi(x,\lambda) := 
\int_{-\infty}^\infty t_+^\lambda\delta(t - f(x))\varphi(x)\,dt.
\]
This is well-defined and coincides with $f_+^\lambda\varphi$ as a 
distribution on $U_f \times \C_+$.  Then we have
\begin{align*}
&
 \int_{-\infty}^\infty t_+^\lambda t\delta(t - f(x))\varphi(x)\,dt  
= \tilde\varphi(x,\lambda+1), 
\\&
\int_{-\infty}^\infty t_+^\lambda\partial_t (\delta(t - f(x))\varphi(x))\,dt
= -\int_{-\infty}^\infty \partial_t(t_+^\lambda)\delta(t - f(x))\varphi(x)\,dt
= -\lambda\tilde\varphi(x,\lambda-1).
\end{align*}
Let $D_{n+1} = D_n\langle t,\partial_t\rangle$ be the $(n+1)$-th Weyl algebra 
with $\partial_t = \partial/\partial t$. 
Let us consider the ring $D_n\langle s,E_s,E_s^{-1}\rangle$ 
of difference-differential operators with the shift operator 
$E_s : s \mapsto s+1$, where $s$ is an indeterminate corresponding to $\lambda$. 
In view of the identities above, let us
define the ring homomorphism (Mellin transform of operators)
\[
\mu : D_{n+1} \longrightarrow D_n\langle s, E_s, E_s^{-1} \rangle
\]
by 
\[
\mu(t) = E_s, \quad \mu(\partial_t) = -sE_s^{-1}, 
\quad \mu(x_j) = x_j, \quad \mu(\partial_{x_j}) = \partial_{x_j}. 
\]
It is easy to see that $\mu$ is well-defined and injective 
since $[\partial_t,t] = [\mu(\partial_t),\mu(t)] = 1$. 
Hence we may regard 
$D_{n+1}$ as a subring of $D_n\langle s, E_s, E_s^{-1} \rangle$. 
Since $\mu(\partial_tt) = -s$, we can also regard $D_n[s]$ as a
subring of $D_{n+1}$. Thus we have inclusions
\[
 D_n[s] \,\,\subset\,\, D_{n+1} \,\,\subset\,\, 
D_n\langle s, E_s, E_s^{-1} \rangle
\]
of rings and $L\otimes_{\C[x]}M$ has a structure of 
left $D_n\langle s,E_s,E_s^{-1}\rangle$-module compatible with that of left $D_n[s]$-module. 
Let $\Fsc(U)$ be the $\C$-vector space of the 
$\Db(U)$-valued meromorphic functions on $\C$.
Then $\Fsc(U)$ has a natural structure of left $D_n\langle s, E_s, E_s^{-1} \rangle$-module, which is compatible with that of $D_n[s]$-module. 
In particular, we can regard $\Fsc(U)$ as a left $D_{n+1}$-module.

\subsection{Computation of $N = D_n[s](f^s\otimes u)$} 

The inclusion $D_{n+1}f^s \subset L = \C[x,f^{-1},s]f^s$ induces a natural 
$D_{n+1}$-homomorphism 
\[
\begin{CD}
D_{n+1}f^s \otimes_{\C[x]}M @>{\iota}>> L \otimes_{\C[x]} M \\
\cup && \cup \\
N'                          @>{\iota'}>> N 
\end{CD}
\]
where $N'$ is the left $D_n[s]$-submodule 
of $D_{n+1}f^s\otimes_{\C[x]}M$ generated by $f^s\otimes u$ and 
$N$ is the left $D_n[s]$-submodule of $L\otimes_{\C[x]}M$ generated by $f^s\otimes u$. 
The homomorphism $\iota$ induces a surjective $D_n[s]$-homomorphism
$\iota' : N' \rightarrow N$. 

\begin{prp}\label{prop:iota}
The homomorphism $\iota$ is injective if and only if 
$M$ is $f$-saturated; i.e., the homomorphism $f : M \rightarrow M$ 
is injective.
\end{prp}

\begin{proof}
First note that $D_{n+1}f^s$ is isomorphic to the first local cohomology 
group 
$\C[x,t,(t-f)^{-1}]/\C[x,t]$ 
of $\C[x,t]$ supported in the non-singular hypersurface 
$t-f(x) = 0$ since
\[
(t-f)f^s = 0,\quad (\partial_{x_i} + f_i\partial_t)f^s = 0 
\quad (i=1,\dots,n). 
\]
In particular, $D_{n+1}f^s$ is a free $\C[x]$-module generated by 
$\partial_t^jf^s$ with $j \geq 0$. 
Hence an arbitrary element $w$ of $D_{n+1}f^s\otimes_{\C[x]}M$ is uniquely written in the form
\[
 w = \sum_{j=0}^k (\partial_t^jf^s) \otimes u_j 
\]
with $u_j \in M$ and $k \in \N$. Then 
\begin{align*}
\iota(w) &= \sum_{j=0}^k (-1)^js(s-1)\cdots (s-j+1)f^{s-j}\otimes u_j
\end{align*}
vanishes if and only if $f^{s-j}\otimes u_j = 0$,
which is equivalent to $f^{m_j}u_j = 0$ with some $m_j \in \N$ 
by Lemma \ref{lemma:saturation}, 
for all $j=0,1,\dots,k$.
This completes the proof.
\end{proof}

Let $\tilde M$ be the left $D_n$-submodule of the 
localization $M[f^{-1}] := \C[x,f^{-1}]\otimes_{\C[x]}M$ 
which is generated by $1\otimes u$. 
Then $\tilde M$ is $f$-saturated and the natural homomorphism
\[
L \otimes_{\C[x]} M  \longrightarrow L\otimes_{\C[x]}\tilde M
\] 
is an isomorphism by Lemma \ref{lemma:saturation}.  

An algorithm to compute $M[f^{-1}]$ was presented in \cite{OTW} under the assumption that $M$ is holonomic on $\C^n \setminus \{f=0\}$. 
It provides us with an algorithm to compute $\tilde M$, i.e., 
the annihilator of $1\otimes u \in M[f^{-1}]$. 
Hence we may assume, from the beginning, 
that $M$ is holonomic and $f$-saturated. 
Then $\iota':N'\rightarrow N$ is an isomorphism by Proposition \ref{prop:iota}.
The $f$-saturatedness is equivalent to the vanishing of the zeroth local 
cohomology group of $M$ with support in $f=0$, which can be computed by  
algorithms presented in \cite{O1},\cite{W},\cite{OT2}.

Thus we have only to give
an algorithm to compute the structure of $N'$ assuming $M$ to be 
$f$-saturated. We  follow an argument introduced by Walther \cite{W}.
Note that we gave in \cite{O1} an algorithm based on tensor product computation 
which is less efficient. 

\begin{dfn}
For a differential operator $P=P(x,\partial) \in D_n$, set
\[
\tau(P) := P\left(x, 
\partial_{x_1} + f_1\partial_{t},\dots,\partial_{x_n} + f_n\partial_{t}
\right) \in D_{n+1}
\]
with $f_j = \partial f/\partial x_j$. 
This substitution is well-defined 
since the operators $\partial_{x_j} + f_j\partial_t$ commute with one another 
and $[\partial_{x_j} + f_j\partial_t,x_i] = \delta_{ij}$ holds. 

Moreover, for a left ideal $I$ of $D_{n+1}$, let $\tau(I)$ be the 
left ideal of $D_{n+1}$ which is generated by the set 
$\{\tau(P) \mid P \in I\}$. 
\end{dfn}

\begin{lmm}\label{lemma:tau}
$\tau(P)(f^s\otimes v) = f^s\otimes (Pv)$ holds in $L\otimes_{\C[x]}M$ 
for any $P \in D_n$ and $v\in M$.
\end{lmm}

\begin{proof}
By the definition of the action of $D_{n+1}$ on $L \otimes_{\C[x]}M$ via 
the Mellin transform, we have
\[
(\partial_{x_j} + f_j \partial_t)(f^s\otimes v) 
= sf^{-1}f_jf^s \otimes v + f^s \otimes(\partial_{x_j}v)
- sf_j f^{-1}f^s \otimes v = f^s \otimes(\partial_{x_j}v). 
\]
This implies the conclusion of the lemma. 
\end{proof}

\begin{prp}
Let $I$ be a left ideal of $D_n$ and set $M = D_n/I$ 
with $u \in M$ being the residue class of $1$ modulo $I$.  
Let $J$ be the left ideal of $D_{n+1}$ which is generated by
$\tau(I) \cup \{ t - f(x)\}$.
Then $J$ coincides with the annihilator
$\Ann_{D_{n+1}}(f^s\otimes u)$ of $f^s\otimes u\in D_{n+1}f^s\otimes_{\C[x]}M$. 
\end{prp}

\begin{proof}
We have only to show that for $P \in D_{n+1}$ the equivalence
\[
 P \in J \quad\Leftrightarrow\quad P(f^s\otimes u) = 0 \quad\mbox{in}\quad 
D_{n+1}f^s \otimes_{\C[x]} M. 
\]
Suppose $Q$ belongs to $J$. Then $P$ annihilates $f^s\otimes u$ by Lemma 
\ref{lemma:tau}.

Conversely, suppose $P(f^s\otimes u) = 0$ in $D_{n+1}f^s \otimes_{\C[x]}M$. 
We can rewrite $P$ in the form 
\[
P = \sum_{\alpha\in\N^n,\nu\in\N}
p_{\alpha,\nu}(x)
\Bigl(\partial_{x_1} + \frac{\partial f}{\partial{x_1}}
\partial_{t} \Bigr)^{\alpha_1}
\cdots
\Bigl(\partial_{x_n} + \frac{\partial f}{\partial{x_n}}
\partial_{t} \Bigr)^{\alpha_n}\partial_t^\nu
+ Q\cdot(t - f(x))
\nonumber
\]
with $p_{\alpha,\nu}(x) \in \C[x]$ and $Q \in D_{n+1}$. 
Setting $P_\nu := \sum_{\alpha \in \N^n}p_{\alpha,\nu}(x)\partial_x^\alpha$,
we get
\[
0 = P(f^s\otimes u ) = \sum_{\nu=0}^\infty (\partial_t^\nu f^s)\otimes P_\nu u
\,\in\, D_{n+1}f^s \otimes_{\C[x]} M. 
\]
It follows that each $P_\nu$ belongs to $I$ 
since $\{\partial_t^\nu f^s\}$ constitutes a free basis of 
$D_{n+1}f^s$ over $\C[x]$. 
Hence we have
\[
 P = \sum_{\nu = 1}^\infty \partial_t^\nu \tau(P_\nu) 
+ Q\cdot(t-f(x))\,\, \in J. 
\]
This completes the proof.
\end{proof}

In order to compute the structure of the $D_n[s]$-submodule $N' = D_n[s](f^s\otimes u)$ 
of $D_{n+1}f^s \otimes_{\C[x]}M$, we have only to compute the annihilator
\[
 \Ann_{D_n[s]}(f^s \otimes u) = D_n[s] \cap J,
\]
where we regard $D_n[s]$ as a subring of $D_{n+1}$.
This can be done as follows: 

Introducing new variables $\sigma$ and $\tau$, 
for $P \in D_{n+1}$, let $h(P) \in D_{n+1}[\tau]$ be the homogenization 
of $P$ with respect to the weights 
\begin{center}
\begin{tabular}{lllll}
\hline
$x_j$ & $\partial_{x_j}$ & $t$ & $\partial_t$ & $\tau$ \\ \hline
$0$     &  $0$               & $-1$ & $1$  & $-1$ \\ \hline
\end{tabular}
\end{center}
Let $J'$ be the left ideal of $D_{n+1}[\sigma,\tau]$ generated by
\[
\{ h(P) \mid P \in \tilde G \} \cup \{ 1 - \sigma\tau \},
\]
where $\tilde G$ is a set of generators of $J$. 

Set $J'' = J' \cap D_{n+1}$. 
Since each element $P$ of $J''$ is homogeneous with respect to the above weights, 
there exists $P'(s) \in D_n[s]$ such that $P = S P'(-\partial_t t)$ 
with $S = t^\nu$ or $S = \partial_t^\nu$ with some integer $\nu \geq 0$.   
We set $P'(s) = \psi(P)(s)$. 
Then $\{ \psi(P) \mid P \in J''\}$ generates the left ideal 
$J \cap D_n[s]$ of $D_n[s]$. This procedure can be done 
by using a Gr\"obner basis in $D_{n+1}[\sigma,\tau]$. 
In conclusion, we have a set of generators of  
$J \cap D_n[s]$.  Then $N'$, and hence $N$ also if $M$ is $f$-saturated, 
 is isomorphic to $D_n[s]/(J \cap D_n[s])$ as 
left $D_n[s]$-module. 

The generalized $b$-function for $f$ and $u$ can be computed as the generator of
the ideal 
\[
\C[s] \cap (\Ann_{D_n[s]} f^s\otimes u  + D_n[s]f)
\]
of $\C[s]$ by elimination via Gr\"obner basis computation in $D_n[s]$.

\subsection{Holonomic systems for the Laurent coefficients of $f_+^\lambda\varphi$}

Let $\lambda_0$ be an arbitrary complex number. 
Our purpose is to compute a holonomic system of which each coefficient of 
the Laurent expansion of $f_+^\lambda\varphi$ is a solution. 

Take $m \in \N$ such that $\myRe\lambda_0 + m > 0$. 
Let $b_0(s)$ be the $b$-function of $f$ and $u$. 
We can find a $P_0(s) \in D_n[s]$ such that
\[
 P_0(s)(f^{s+1}\otimes u) = b_0(s)f^s\otimes u
\]
holds in $N$ by, e.g., syzygy computation.   
By using this functional equation, 
we can find a nonzero polynomial $b(s)$ and 
$P(s) \in D_n[s]$ such that
\[
b(\lambda)f_+^{\lambda} = P(\lambda) f_+^{\lambda+m}. 
\]
In fact, we have only to set 
\[
P(s) := P_0(s)P_0(s+1)\cdots P_0(s+m-1),
\quad
b(s) := b_0(s)b_0(s+1)\cdots b_0(s+m-1). 
\]
Factorize $b(s)$ as $b(s) = c(s)(s-\lambda_0)^l$ 
with $c(\lambda_0) \neq 0$. 
Then $f_+^\lambda\varphi$ has a Laurent expansion of the form
\[
 f_+^{\lambda}\varphi = \sum_{k=-l}^\infty (\lambda-\lambda_0)^k \varphi_k
\]
around $\lambda_0$, where  $\varphi_k \in \Db(U)$ is given by
\begin{equation}
\varphi_k 
=  \frac{1}{(l+k)!} \lim_{\lambda\rightarrow \lambda_0}
\left(\frac{\partial}{\partial\lambda}\right)^{l+k}
(c(\lambda)^{-1}P(\lambda)f_+^{\lambda+m}) 
\nonumber\\ 
=\sum_{j=0}^{l+k}Q_{kj} (f_+^{\lambda_0+m}(\log f)^j)
\label{eq:u_k}
\end{equation}
with 
\[
Q_{kj} := \frac{1}{j!(l+k-j)!}
\left[\left(\frac{\partial}{\partial\lambda}\right)^{l+k-j}
(c(\lambda)^{-1}P(\lambda))\right]_{\lambda = \lambda_0}.
\]

Let
\[
\tilde L = \C[x,f^{-1},s]f^s \oplus\C[x,f^{-1},s]f^s\log f 
\oplus \C[x,f^{-1},s]f^s(\log f)^2 \oplus \cdots
\]
be the free $\C[x,f^{-1},s]$-module with a natural structure of 
left $D_n\langle s,\partial_s\rangle$-module. 
Consider the left $D_n[s]$-submodule
\[
N[k] = D_n[s](f^s\otimes u) + D_n[s]((f^s\log f)\otimes u) 
+ \cdots + D_n[s]((f^s(\log f)^k)\otimes u)
\]
of $\tilde L \otimes_{\C[x]}M$. 
For a complex number $\lambda_0$, set 
\[
N_{\lambda_0}[k] = N[k]/(s-\lambda_0)N[k].  
\]
Let us first give an algorithm to compute the structure of $N[k]$. 

\begin{prp}
\label{prp:fslog}
Let $G_0$ be a set of generators of the annihilator 
$\Ann_{D_n[s]}(f^{s}\otimes u) = J \cap D_n[s]$.  
Let $e_1 = (1,0,\dots,0)$, $\cdots$, $e_{k+1} = (0,\dots,0,1)$ be 
the canonical basis of $\Z^{k+1}$. 
For each $Q(s) \in G_0$ and an integer $j$ with $0 \leq j \leq k$, 
set
\[
Q^{(j)}(s) := \sum_{i=0}^j \binom{j}{i} 
\frac{\partial^{j-i}Q(s)}{\partial s^{j-i}}e_{i+1}
\in (D_n[s])^{k+1}. 
\]
Let $J_k$ be the left $D_n[s]$-submodule of $(D_n[s])^{k+1}$ generated by 
$G_1 := \{Q^{(j)}(s)(\lambda_0) \mid Q(s) \in G_0,\,0 \leq j \leq k\}$.
Then $(D_n[s])^{k+1}/J_k$ is isomorphic to $N[k]$. 
\end{prp}

\begin{proof}
Let $\varpi : (D_n[s])^{k+1} \rightarrow N[k]$ be the canonical 
surjection.  
Let $Q(s)$ belong to $G_0$. 
Differentiating the equation $Q(s)(f^s\otimes u)=0$ in $N[k]$ 
with respect to $s$, one gets 
\[
\sum_{i=0}^j \binom{j}{i} 
\frac{\partial^{j-i}Q(s)}{\partial s^{j-i}}((f^s(\log f)^i) \otimes u)
= 0. 
\]
Hence $J_k$ is contained in the kernel of $\varpi$. 
Conversely, assume that 
\[
\vec{Q}(s) = (Q_0(s),Q_1(s),\dots,Q_k(s))
\]
belongs to the kernel of $\varpi$. This implies 
$Q_k(s)(f^s\otimes u) = 0$ since $N[k]/N[k-1]$ is isomorphic to 
$N = D_n[s](f^s\otimes u)$. Hence $\vec{Q}(s) - Q_k^{(k)}(s)$ 
belongs to the kernel of $\varpi$, the last component of which is zero. 
We conclude that $\vec{Q}(s)$ belongs to $J_k$ 
by induction.
\end{proof}

Thus we have 
\[
N_{\lambda_0}[k] = (D_n)^{k+1}/J_k|_{s=\lambda_0}, 
\qquad J_k|_{s=\lambda_0} := \{Q(\lambda_0) \mid Q(s) \in J_k\}.
\]
Set 
\[
w := \sum_{j=0}^{l+k} Q_{kj}((f^{\lambda_0+m}(\log f)^j)\otimes u), 
\qquad
M_{k} := D_nw.
\]
Then we have
\[
Pw = 0 \quad\Leftrightarrow\quad 
P(Q_0,Q_1,\dots,Q_{l+k}) \in J_{l+k}|_{s=\lambda_0+m}. 
\]
Thus we can find a set of generators of $\Ann_{D_n}w$  by computation of 
syzygy or intersection. 
As was shown in \S 2.4, $\varphi_k$ is a solution of 
the holonomic system $M_k$.

\subsection{Difference equations for the local zeta function}

In the sequel, we assume that $\varphi$ is a locally integrable 
function on $\R^n$. 
As we have seen so far, $f_+^\lambda\varphi \in \Fsc(\R^n)$ is 
a solution of the holonomic $D_{n+1}$-module $D_{n+1}/J$. 
Hence if the local zeta function 
$Z(\lambda) := \int_{\R^n} f_+^\lambda \varphi \,dx$ is well-defined, 
e.g., if $\varphi$ has compact support, or else is smooth on $\R^n$ 
with all its derivatives 
rapidly decreasing on the set $\{x \in \R^n \mid f(x) \geq 0\}$, 
then $Z(\lambda)$ is a solution of the integral module
\[
D_{n+1}/(J + \partial_{x_1}D_{n+1} + \cdots + \partial_{x_n}D_{n+1})
\]
of $D_{n+1}/J$, which is a holonomic module over $D_1 = \C\langle t,\partial_t \rangle$. 
This $D_1$-module can be computed by the integration algorithm which is the 
`Fourier transform' 
of the restriction algorithm given in \cite{OT2} 
(see \cite{OT1} for the integration algorithm). 
Then by Mellin transform we obtain linear 
difference equations for $Z(\lambda)$. 
Thus we get

\begin{thm}
Under the above assumptions, 
$Z(\lambda)$ satisfies a non-trivial linear difference equation 
with polynomial coefficients in $\lambda$.  
\end{thm}

\begin{exa}
$\Gamma(\lambda +1) = \int_0^\infty x^{\lambda}e^{-x}\,dx 
= \int_{-\infty}^\infty x_+^{\lambda}e^{-x}\,dx$ 
satisfies the difference equation
\[
(E_\lambda - (\lambda+1))\Gamma(\lambda+1) = 0,
\]
where $E_\lambda : \lambda \mapsto \lambda+1$ is the shift operator. 
\end{exa}

\subsection{Examples}

Let us present some examples computed by using algorithms introduced so far 
and their implementation in the computer algebra system Risa/Asir. 

\begin{exa}
Set $f = x^3-y^2 \in \R[x,y]$ and $\varphi = 1$. Since the $b$-function of $f$ is 
$b_f(s)=(s+1)(6s+5)(6s+7)$, possible poles of $f_+^\lambda$ are 
$-1-\nu$, $-5/6-\nu$, $-6/7-\nu$ with $\nu\in\N$ and they are at most simple poles. 
The residue $\Res_{\lambda=-1}f_+^\lambda$ is a solution of 
\[
D_2/(D_2(2x\partial_x +3y\partial_y+6)
+ D_2(2y\partial_x + 3x^2\partial_y) + D_2(x^3-y^2)).
\]
$\Res_{\lambda=-5/6}f_+^\lambda$ is a solution of $D_2/(D_2x + D_2y)$. Hence it 
is a constant multiple of the delta function $\delta(x,y)= \delta(x)\delta(y)$. 
$\Res_{\lambda=-7/6}f_+^\lambda$ is a solution of 
$D_2/(D_2x^2 + D_2(x\partial_x+2) + D_2y)$. Hence it is a constant multiple 
of $\delta'(x)\delta(y)$.
\end{exa}

\begin{exa}
Set $f = x^3-y^2$ and $\varphi(x,y) = \exp(-x^2-y^2)$. 
Then $\varphi$  is a solution of a holonomic system
$M := D_2/(D_2(\partial_x+2x) + D_2(\partial_y+2y))$ on $\R^2$,
which is $f$-saturated since it is a simple $D_2$-module. 
The generalized $b$-function for $f$ and $u := [1] \in M$ is 
$b_f(s) = (s+1)(6s+5)(6s+7)$. 
The local zeta function 
$Z(\lambda) := \int_{\R^2} f_+^\lambda\varphi\,dxdy$ is annihilated by
the difference operator
\begin{align*}&
32E_s^4 + 16(4s+13)E_s^3 
-4(s+3)(27s^2+154s+211)E_s^2 
\\&
-6(s+2)(s+3)(36s^2+162s+173)E_s
-3(s+1)(s+2)(s+3)(6s+5)(6s+13),
\end{align*}
where $s$ is an indeterminate corresponding to $\lambda$. 
From this we see that $-7/6$ is not a pole of $Z(\lambda)$. 
\end{exa}

\begin{exa}
Set $\varphi(x) = \exp(-x - 1/x)$ for $x>0$ and $\varphi(x) = 0$ for $x \leq 0$. 
Then $\varphi(x)$ belongs to the space $\Ssc(\R)$ of rapidly decreasing functions 
on $\R$ and satisfies a holonomic system 
\[
M := 
D_1/D_1(x^2\partial_x + x^2-1),
\]
which is $x$-saturated. 
The generalized $b$-function for $f = x$ and $u = [1] \in M$ 
is $s+1$. 
The local zeta function 
$Z(\lambda) := \int_{\R} x_+^\lambda\varphi(x)\,dx$ is entire (i.e., 
without poles) and  
satisfies a  difference equation
\[
(E_\lambda^2 - (\lambda+2)E_\lambda - 1)Z(\lambda) = 0. 
\]
This can also be deduced by integration by parts. 
\end{exa}

\begin{exa}
Set $\varphi_1(x) = \exp(-x - 1/x)$ for $x>0$ and $\varphi_1(x) = 0$ for $x \leq 0$. 
Set $\varphi(x,y) = \varphi_1(x)e^{-y}$. 
Then $\varphi$ satisfies a holonomic system
\[
M := D_2/(D_2(x^2\partial_x + x^2-1) + D_2(\partial_y+1)). 
\] 
The generalized $b$-function for $f := y^3-x^2$ and $u = [1] \in M$ 
is $s+1$. Moreover, we can confirm that $M$ is $f$-saturated by using 
the localization algorithm in \cite{OTW}. 
The local zeta function 
$Z(\lambda) := \int_{\R^2} f_+^\lambda\varphi\,dxdy$  
is well-defined since $f(x,y) < 0$ if $y < 0$. It is annihilated by a 
 difference operator of the form
\begin{align*}&
E_s^{11} + a_{10}(s)E_s^{10} + \cdots + a_1(s)E_s + a_0(s), 
\\&
a_0(s) = c(s+1)(s+2)(s+3)(s+4)(s+5)(s+6)(s+7)(s+8)(s+9),
\end{align*}
where $c$ is a positive rational number and $a_1(s),\dots,a_{10}(s)$ are 
polynomials of $s$ with rational coefficients.
Possible poles of $f_+^\lambda\varphi$ are the negative integers. 
For example, $-1$ is at most a simple pole of $f_+^\lambda\varphi$ 
and $\Res_{\lambda=-1} f_+^\lambda\varphi$ is a solution of a holonomic system
\[
D_2/(D_2(3x^2\partial_x + 2xy\partial_y + 3x^2 + (2y+6)x - 3) + D_2(y^3-x^2)). 
\]

\end{exa}

\begin{exa}
Set $f = x^3-y^2z^2$. The $b$-function of $f$ is  
$(s+1)(3s+4)(3s+5)(6s+5)^2(6s+7)^2$. 
For example, its maximum root $-5/6$ is at most a pole of order $2$ of $f_+^\lambda$. 
Let 
\[
f_+^\lambda = \Bigl(\lambda + \frac56\Bigr)^{-2}\varphi_{-2} 
+ \Bigl(\lambda + \frac56\Bigr)^{-1}\varphi_{-1} + \varphi_0 + \cdots 
\]
be the Laurent expansion. Then $\varphi_{-2}$ satisfies
\[
x\varphi_{-2} = y \varphi_{-2} = z\varphi_{-2} = 0.
\]
Hence $\varphi_{-2}$ is a constant multiple of $\delta(x,y)$. 
On the other hand, $\varphi_{-1}$ satisfies a holonomic system
\[
x\varphi_{-1} = (y\partial_y - z\partial_z)\varphi_{-1} 
= yz\varphi_{-1} = (z^2\partial_z-z)\varphi_{-1} = 0.
\]
\end{exa}

\end{document}